\documentclass[11pt]{article}
\usepackage{tikz}
\usepackage{subfigure}
\usepackage[english]{babel}

\usepackage{graphicx}
\usetikzlibrary{arrows, graphs}
\usepackage{amsfonts,amssymb,amsmath,latexsym,amsthm}
\usepackage{multirow}
\usepackage{bbm}

\usepackage{thmtools}
\usepackage{thm-restate}

\usepackage[colorlinks=true,allcolors=blue,backref=page]{hyperref}
\usepackage[noabbrev,capitalize,nameinlink]{cleveref}

\newcommand \BB{\mathcal B}
\newcommand \HH{\mathcal H}
\newcommand\B{\text{B-}}

\newtheorem{thm}{Theorem}[section]
\newtheorem{mydef}[thm]{Definition}

\newtheorem{lem}[thm]{Lemma}
\newtheorem{prop}[thm]{Proposition}
\newtheorem{prob}[thm]{Problem}

\newtheorem{obs}[thm]{Observation}
\newtheorem*{claim}{Claim}

\begin{document}

\title{Ramsey-type results on parameters related to domination }
\author{Jin Sun$^a$, \quad Xinmin Hou$^{a,b,c}$\footnote{\text{e-mail: jinsun@mail.ustc.edu.cn (J. Sun), xmhou@ustc.edu.cn (X. Hou)}}\\
	\small $^{a}$ School of Mathematical Sciences\\
	\small University of Science and Technology of China, Hefei, Anhui 230026, China.\\
	\small  $^{b}$ CAS Key Laboratory of Wu Wen-Tsun Mathematics\\
	\small University of Science and Technology of China, Hefei, Anhui 230026, China.\\
	\small$^c$ Hefei National Laboratory,\\
	\small University of Science and Technology of China, Hefei 230088, Anhui, China
}
\date{}
\maketitle
\begin{abstract}
The following  inequality chain $$		ir(G)\le \gamma(G)\le i(G)\le \alpha(G) \le \Gamma(G) \le I\!R(G)$$
is known as a domination chain, where $ir(G), \gamma(G), i(G), \alpha(G), \Gamma(G)$, and $I\!R(G)$ are the lower irredundance number, the domination number, the independence domination number, the independence number, the upper domination number and the upper irredundance number of $G$, respectively.
The Ramsey-type problem seeks  to characterize the family $\HH$ of graphs such that every $\HH$-free graph $G$ has a bounded parameter $\mu$. The classical Ramsey's theorem states that every $\{K_n, E_n\}$-free graph has a bounded number of vertices. Furuya (Discrete Math.Theor 2018) characterized $\HH$ such that every connected $\HH$-free graph $G$ has a bounded domination number. The characterization  of the graph family $\HH$ for which every connected $\HH$-free graph $G$ has a bounded independence number was due to Choi, Furuya, Kim, Park~(Discrete math. 2020) and Chiba, Furuya (Electron. J. Combin., 2022).  In this paper, we further characterize $\HH$  such that every connected $\HH$-free graph $G$ has bounded  $\mu(G)$ for $\mu$ belonging to the set $\{ir(G), i(G), \Gamma(G), \text{IR}(G)\}$. This completes the characterization of $\HH$ for which every connected $\HH$-free graph $G$ has bounded $\mu(G)$ for $\mu(G)$ along the domination chain. 
Additionally, we characterize $\HH$ such that every connected $\HH$-free graph $G$ has bounded $\mu(G)$ for $\mu$ related to the domination number.  Specifically, we consider the parameters  $O\!I\!R(G)$, $I\!S(G)$, or $I\!R\!S(G)\}$, where $O\!I\!R(G)$, $I\!S(G)$, and $I\!R\!S(G)$ are the open irredundance number, the independence saturation number, and the irredundance saturation number of  graph $G$, respectively.  
				
\end{abstract}
\textbf{Keywords:} {Ramsey-type problem, domination number, domination chain}

\section{Introduction}
%For positive integers $m$ and $n$ with $m<n$, let $[m,n]$ denote the set $\{m,m+1,\dots ,n-1,n\}$, and specially let $[n]$ denote $[1,n]$. %For a real number $x$, it shoulded be noted that we use the symbol $\lceil x\rceil$ to denote the {\it ceiling function} of $x$. 
In this paper, we consider only finite and simple graphs. Let $G=(V,E)$ be a graph. For a vertex $v\in V$, let $N_G(v)=\{u\in V: uv\in E\}$ denote the {\it (open) neighborhood} of $v$, and let $N_G[v]=N(v)\cup \{v\}$ denote the {\it closed neighborhood} of $v$ (the subscript may be omitted if there is no confusion). For a subset $S\subseteq V$, let $N_G[S]=\cup_{v\in S}N_G[v]$.
We denote by $G[S]$ the subgraph of $G$ induced by $S$. A set $S\subseteq V$ is called a {\it stable set} (or an {\it independent set}) of $G$ if there is no edge in $G[S]$. 
%For two subsets $S$ and $T$ of $V(G)$, we say $S$ is {\it anticomplete} to $T$ if there is no edge between vertices in $S$ and $T$. 
%The {\em center} of a graph $G$ is a vertex $c\in V(G)$ with $d(c,x)\le \lceil\frac{diam(G)}2\rceil$ for any vertex $x\in V(G)$. 
%Clearly, a star $K_{1,t}$ has a unique center, and a path $P_n$ has a unique center if $n$ is odd and two centers if $n$ is even. 
The more notation, see \cite{d}.	

For graphs $H_1$ and $H_2$, we say $H_1\prec H_2$ if $H_1$ is an induced subgraph of $H_2$. Let $\HH$ be a family of graphs, we say $G$ is $\HH$-free, if there is no graph $H\in \HH$ such that $H\prec G$. 
%In this case, $\HH$ is called the forbidden subgraphs. 
For graph families $\HH_{1}$ and $\HH_2$, we say $\HH_1\le \HH_2 $ if for any $H_2\in \HH_2$ there exists $H_1\in \HH_1$ such that $ H_1\prec H_2$, i.e., each graph in $\HH_2$ is not $\HH_1$-free. 
%Specially, $H_1\prec H_2$ is equivalent to $\{H_1\}\le \{H_2\}$.
The following straightforward result will be commonly used without explicit mention.
\begin{obs}
	The relation `$\le$' is transitive. Therefore, if $\mathcal{H}_1\le \mathcal{H}_2$, then every $\mathcal{H}_1$-free graph is also $\mathcal{H}_2$-free. 
\end{obs} 

%For positive integers $a$ and $b$, let $R(a,b)$ denote the Ramsey number, i.e., $R(a,b)$ is the minimum integer to insure that there is a clique of order $a$ or a stable set of order $b$ in any graph of order $R(a,b)$. 
As usual, let $P_n,\, C_n,\, K_n$ and $E_n$ denote a {\em path}, a {\em cycle}, a {\em complete graph},  and an {\em empty graph} of order $n$, respectively.
Let $K_{s,t}$ be the {\em complete bipartite graph} with partitions of orders $s$ and $t$.  
When $s=1$, we call  $K_{1,t}$ a {\em star}.  
%Now we focus on Ramsey-type problems. Using this language, 

The classical Ramsey's Theorem can be stated as follows: 

(A) (Ramsey's Theorem~\cite{r}, 1929) For a family $\HH$ of graphs, there is a constant $c=c(\HH)$ such that $|V(G)|< c$ for every $\HH$-free graph $G$ if and only if $\HH \le \{K_m,E_n\}$ for some positive integers $m, n$;

(B) (The connected version, Proposition 9.4.1 in \cite{d}) For a family $\HH$ of graphs, there is a constant $c=c(\HH)$ such that $|V(G)|< c$ for every connected $\HH$-free graph $G$ if and only if $\HH \le \{K_n, K_{1,n}, P_n\}$ for some positive integer $n$. 

The {\em Ramsey number} $R(m,n)$ is the least integer $c=c(\{K_m, E_n\})$ in (A), i.e., every graph $G$ of order at least $R(m,n)$ contains either a complete graph $K_m$ or an empty graph $E_n$.  
%for a family $\HH$ of graphs, there is a constant $c=c(\HH)$ such that $|V(G)|< c$ for every $\HH$-free graph $G$ if and only if $\HH \le \{K_n,E_n\}$ for some positive integer $n$. Furthermore, the connected version says that for a family $\HH$ of graphs, there is a constant $c=c(\HH)$ such that $|V(G)|< c$ for every connected $\HH$-free graph $G$ if and only if $\HH \le \{K_n, K_{1,n}, P_n\}$ for some positive integer $n$. The general philosophy is as follows. 

In general, given a graph parameter $\mu$, define
\[
B\text{-}\mu=\{\HH: \text{there is constant } c \text{ such that } \mu(G)<c \text{  for any connected } \HH \text{-free graph } G\}.
\] 
A Ramsey-type problem of $\mu$ is to determine $B$-$\mu$.
\begin{prob}\label{PROB: mu}
	Given a graph parameter $\mu$, determine $\B\mu$.	
\end{prob}

In the following, we list more results on different parameters of \cref{PROB: mu}. %up to now
\begin{description}
%	\item[(1)]	The connected version of the Ramsey theorem ((B) in Theorem~\ref{THM: Ramsey}) can be rephrased as $\HH\in\B$ord if and only if $\HH\le\{K_n, K_{1,n}, P_n\}$.

	\item[(1)] Chiba and Furuya  determined $\B\mu$ when $\mu$ is the path cover/partition number in \cite{cf22}, and when $\mu$ is the induced star and path cover/partition number in \cite{cf23}.
	
	\item[(2)] Choi, Furuya, Kim and Park \cite{cfkp20} determined $\B\nu$, where $\nu(G)$ is the matching number of $G$.
	
%	\item[(4)] Furuya \cite{f18}  determined $\HH\in\B\gamma$, where $\gamma$ is the domination number of a graph.
	\item[(3)] Lozin \cite{l17} determined $\B\mu$ when $\mu$ is the neighborhood diversity or VC-dimension.

	% (two parameters concerning the property of the closed neighborhood of vertices, one can refer to \cite{l17} for details).
	\item[(4)] Lozin and Razgon \cite{lr22} determined $\B\mu$, where $\mu(G)$ is the tree-width of graph $G$.
	%, which is the smallest integer such that graph $G$ has a tree-decomposition into parts with at most $k+1$ vertices.
	
	\item[(5)] Kierstead and Penrice \cite{kp}; and Scott, Seymour, and Spirkl \cite{sss23} determined $\B\delta$, where $\delta(G)$ is the minimal degree of graph $G$.
	
	\item[(6)] Galvin, Rival and Sands \cite{grs}; and Atminas, Lozin and Razgon \cite{alr} determined $\B\mu$, where $\mu(G)$ is the length of the longest path of graph $G$.
	
	\item[(7)] Sun, and Hou~\cite{SH23} determined $\B\mu$, where $\mu(G)$ is the deficiency of graph $G$.
\end{description}

Let $nG$ be the graph consisting of $n$ disjoint copies of $G$. In order to state the forbidden subgraphs condition conveniently, we introduce additional kinds of graphs (as shown in~\cref{fig 1}). 

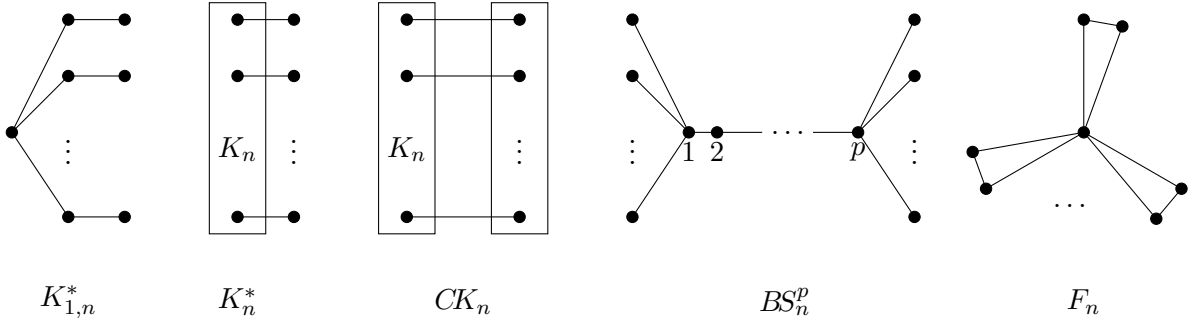
\begin{figure}
	\centering
	\begin{tikzpicture}[baseline=10pt,scale=0.75]
		\filldraw[fill=black,draw=black] (0,0) circle (0.1);
		\filldraw[fill=black,draw=black] (1,2) circle (0.1);
		\filldraw[fill=black,draw=black] (1,1) circle (0.1);
		\node at (1,-0.2) {$\vdots$};
		\filldraw[fill=black,draw=black] (1,-1.5) circle (0.1);
		\filldraw[fill=black,draw=black] (2,2) circle (0.1);
		\filldraw[fill=black,draw=black] (2,1) circle (0.1);
		\filldraw[fill=black,draw=black] (2,-1.5) circle (0.1);
		\draw (0,0) -- (1,2);
		\draw (0,0) -- (1,1);
		\draw (0,0) -- (1,-1.5);
		\draw (2,2) -- (1,2);
		\draw (2,1) -- (1,1);
		\draw (2,-1.5) -- (1,-1.5);
		\node at (1,-3) {$K_{1,n}^*$};
		
		\draw (3.5,-1.8) rectangle (4.5,2.3);
		\filldraw[fill=black,draw=black] (4,-1.5) circle (0.1);
		\filldraw[fill=black,draw=black] (4,1) circle (0.1);
		\filldraw[fill=black,draw=black] (4,2) circle (0.1);
		\node at (5,-0.2) {$\vdots$};
		\node at (4,-0.3) {$K_n$};
		\filldraw[fill=black,draw=black] (5,-1.5) circle (0.1);
		\filldraw[fill=black,draw=black] (5,1) circle (0.1);
		\filldraw[fill=black,draw=black] (5,2) circle (0.1);
		\draw (4,-1.5) -- (5,-1.5);
		\draw (4,1) -- (5,1);
		\draw (4,2) -- (5,2);
		\node at (4,-3) {$K_{n}^*$};
		
		\draw (7.5,-1.8) rectangle (6.5,2.3);
		\draw (8.5,-1.8) rectangle (9.5,2.3);
		\filldraw[fill=black,draw=black] (7,-1.5) circle (0.1);
		\filldraw[fill=black,draw=black] (9,-1.5) circle (0.1);
		\filldraw[fill=black,draw=black] (7,1) circle (0.1);
		\filldraw[fill=black,draw=black] (9,1) circle (0.1);
		\filldraw[fill=black,draw=black] (7,2) circle (0.1);
		\filldraw[fill=black,draw=black] (9,2) circle (0.1);
		\draw (7,-1.5) -- (9,-1.5);
		\draw (7,1) -- (9,1);
		\draw (7,2) -- (9,2);
		\node at (9,-0.2) {$\vdots$};
		\node at (7,-0.3) {$K_n$};
		\node at (8,-3) {$C\!K_n$};

		\begin{scope}[shift={(-0.5,0)}]
			\filldraw[fill=black,draw=black] (11.5,-1.5) circle (0.1);
			\filldraw[fill=black,draw=black] (11.5,1) circle (0.1);
			\filldraw[fill=black,draw=black] (11.5,2) circle (0.1);
			\filldraw[fill=black,draw=black] (12.5,0) circle (0.1);
			\node at (11.5,-0.2) {$\vdots$};
			
			\filldraw[fill=black,draw=black] (16.5,-1.5) circle (0.1);
			\filldraw[fill=black,draw=black] (16.5,1) circle (0.1);
			\filldraw[fill=black,draw=black] (16.5,2) circle (0.1);
			\filldraw[fill=black,draw=black] (15.5,0) circle (0.1);
			\node at (16.5,-0.2) {$\vdots$};
			
			\filldraw[fill=black,draw=black] (13,0)node[below]{$2$} circle (0.1);
			%\filldraw[fill=black,draw=black] (13.5,0) circle (0.1);
			%\filldraw[fill=black,draw=black] (15,0) circle (0.1);
			\draw (11.5,-1.5) -- (12.5,0);
			\draw (11.5,1) -- (12.5,0);
			\draw (11.5,2) -- (12.5,0);
			\draw(12.5,0)node[below]{$1$} -- (13.8,0);
			\draw(14.7,0) -- (15.5,0);
			\draw (16.5,-1.5) -- (15.5,0)node[below]{$p$};
			\draw (16.5,1) -- (15.5,0);
			\draw (16.5,2) -- (15.5,0);
			%\node at (13,-0.5) {$2$};
			
			\node at (14.3,0) {$\dots$};
			%\node at (15,-0.5) {$p$};
			\node at (14.2,-3) {$B\!S_n^p$};	
		\end{scope}
		
		\filldraw[shift ={(19,0)}][fill=black,draw=black] (0:0) circle (0.1);
		\filldraw[shift ={(19,0)}][fill=black,draw=black] (90:2) circle (0.1);
		\filldraw[shift ={(19,0)}][fill=black,draw=black] (70:2) circle (0.1);
		\filldraw[shift ={(19,0)}][fill=black,draw=black] (210:2) circle (0.1);
		\filldraw[shift ={(19,0)}][fill=black,draw=black] (190:2) circle (0.1);
		\filldraw[shift ={(19,0)}][fill=black,draw=black] (330:2) circle (0.1);
		\filldraw[shift ={(19,0)}][fill=black,draw=black] (310:2) circle (0.1);
		\node at (18.8,-1.3) {$\dots$};
		\node at (19,-3){$F_n$};
		\draw[shift ={(19,0)}] (0:0)--(90:2)--(70:2)--cycle;
		\draw [shift ={(19,0)}](0:0)--(210:2)--(190:2)--cycle;
		\draw[shift ={(19,0)}] (0:0)--(330:2)--(310:2)--cycle;
	\end{tikzpicture}
	\caption{The graphs $K_{1,n}^*,\, K_n^*,\, C\!K_n,\,B\!S_n^p$ and $F_n$.}
	\label{fig 1}
\end{figure}

\begin{itemize}
	
	\item $K_{1,n}^{*} $: The graph obtained by adding a pendant to every leaf of $K_{1,n}$.
	\item $K_n^{*}$: The graph obtained by adding a pendant to every vertex of $K_n$.
	\item $C\!K_n$: The graph obtained by adding a perfect matching between two disjoint copies of $K_n$.
	\item $B\!S_n^p$: The graph obtained by connecting the {\it centers} of two disjoint $K_{1,n}$ by a path $P=P_p$. Specially, when $p=2$, we denote it by $B\!S_n$. We call $B\!S_n^p$ a {\em bistar} and call the two ends of $P_p$ as the centers of  $B\!S_n^p$. 
	\item $F_n$: The graph obtained by adding a universal vertex adjacent to each vertex of $nK_2$.
	
\end{itemize}

In this paper, we mainly concern Problem~\ref{PROB: mu} when $\mu$ are parameters related to domination.

\subsection{The domination chain}
Let $G=(V,E)$ be a graph. For a property $P$ about subsets of $V$, we define a subset $S$ as a {\em minimal} (resp. {\em maximal}) $P$-set if $S$ satisfies $P$, and   no  proper subset (resp. superset)  of $S$ satisfies $P$. 
A set $S\subseteq V$  is called a {\em dominating set} of $G$ %if for any $v\in V\setminus S$, there is a vertex $s\in S$ adjacent to $v$. 
if $V=N[S]$.
Define the {\em domination number} of $G$ as
 $$\gamma(G)=\min\{|S| : S \text{ is a minimal dominating set of } G\},$$ and 
the {\em upper domination number} of $G$ as
$$\Gamma(G)=\max\{|S| : S \text{ is a minimal dominating set of } G\}.$$
Recall that the {\em independence number} of $G$ 
$$\alpha(G)=\max\{|S| : S \text{ is a maximal independent set of $G$}\}.$$ 
Define the {\em independence domination number} of $G$ as 
$$i(G)=\min\{|S| : S \text{ is a maximal independent set of $G$}\}.$$ 

Let $S\subseteq V(G)$ and $v\in S$. The {\em private neighborhood} of $v$ with respect to $S$ is defined as $P\!N[v,S]=N_G[v]-N_G[S\setminus\{v\}]$. %For $T\subseteq S$, write $P\!N[T, S]=\bigcup_{v\in T}P\!N[v, S]$ and $P\!N^{-1}[T, S]=T$. For a vertex $v'\in P\!N[v, S]$, denote by $P\!N^{-1}[v', S]=v$. For $T'\subseteq P\!N[S, S]$, write $P\!N^{-1}[T', S]=\bigcup_{v\in T'}P\!N^{-1}[v', S]$.
A set $S\subseteq V(G)$ is called an {\em irredundant set} if every vertex $s$ in $S$ has at least one private neighbor, i.e., $P\!N[s,S]\neq \emptyset$ for any $s\in S$. 

The {\em upper  irredundance number}, $\text{IR}(G)$, and the {\em lower  irredundance number}, $ir(G)$, of $G$ are  defined as 
$$\mbox{IR}(G)=\max\{|S| : S \text{ is a maximal irredundant set of $G$} \}$$ and 
$$ir(G)=\min\{|S| : S \text{ is a maximal irredundant set of $G$} \}.$$

%These six parameters can be connected by an inequality chain as follows.
It can be checked directly that a maximal stable set is a minimal dominating set, and a minimal dominating set is a maximal irredundant set. 
Visually, %Therefore,
\[\{\mbox{maximal stable set}\}\subseteq \{\mbox{minimal dominating set}\}\subseteq \{\mbox{maximal irredundant set}\}.\]
By the definitions of $ir(G), \gamma(G), i(G), \alpha(G), \Gamma(G)$, and $\text{\it IR}(G)$, we have the following  inequality chain.
\begin{prop}[\cite{hhs}] \label{domination chain}
For any graph $G$, 
\begin{equation}\label{EQ: domination-chain}
	  ir(G)\le \gamma(G)\le i(G)\le \alpha(G) \le \Gamma(G) \le \text{IR}(G).
\end{equation}
\end{prop}

The inequality chain above %in \cref{domination chain} 
is called the {\em domination chain}. For the history and more information of these parameters, the reader can survey the book \cite{hhs} and references therein. For parameters $\mu$ defined by size of special subsets of $V(G)$, we will use $\mu$-set to mean subset $S$ realizing parameter $\mu$; for example, a $\gamma$-set $D$ of $G$ is a minimum dominating set of $G$ such that $\gamma(G)=|D|$.

We first focus on Problem~\ref{PROB: mu} for $\mu$ in the domination chain. Furuya \cite{f18} has determined $B$-$\gamma$. 
\begin{thm}[\cite{f18}]\label{domination number}%\label{THM: domonation}
	$\HH\in B$-$\gamma$  if and only if $\HH \le \{K_{1,n}^*, K_n^*, P_n\}$ for some positive integer $n$.
\end{thm}
Denote by $\gamma_n$ the smallest constant depending only on $n$ in \cref{domination number}, i.e., $\gamma(G)<\gamma_n$ for any $\{K_{1,n}^*, K_n^*, P_n\}$-free graph $G$.

Choi, Furuya, Kim, Park~\cite{cfkp20} and Chiba, Furuya \cite {cf22} characterized $B$-$\alpha$.

\begin{thm}[\cite{cfkp20,cf22}]\label{duli}
	$\HH\in B$-$\alpha$  if and only if $\HH \le \{K_{1,n},K_n^*, P_n\} $ for some positive integer $n$. 	
\end{thm}
%We can get this result directly from \cref{domination number} by considering how an $\alpha$-set is dominated by a $\gamma$-set. 
Denote by $\alpha_n$ the smallest constant depending only on $n$ in Theorem~\ref{duli}, i.e., $\alpha(G)<\alpha_n$ for any $\{K_{1,n}, K_n^*, P_n\}$-free graph $G$.

We proceed to determine $B\text{-}\mu$ for $\mu$ belonging to the set $\{ir(G), i(G), \Gamma(G), \text{IR}(G)\}$. This completes the characterization of $B$-$\mu$ for $\mu$ along the domination chain~(\ref{EQ: domination-chain}). 
%We now consider the forbidden subgraphs condition to bound these parameters in the domination chain. 
\begin{thm}\label{THM: main}
	\begin{itemize}
		\item[(1)] $\HH\in $  $B$-$ir$  if and only if $\HH \le \{K_{1,n}^*, K_n^*, P_n\}$ for some positive integer $n$.	
		\item [(2)] 	$\HH\in B$-$i$  if and only if $\HH \le \{K_{1,n}^*, K_n^*, P_n, K_{n,n}, B\!S_n\}$ for some positive integer $n$.	
		\item[(3)] $\HH\in B$-$\text{IR}$ if and only if $\HH\in B$-$\Gamma$ if and only if $\HH \le \{K_{1,n}, K_n^*, P_n, C\!K_n\}$ for some positive integer $n$.
		%$\HH \le \{K_{1,n}^*, K_n^*, P_n\}$ for some positive integer $n$.
	\end{itemize}

\end{thm}

\subsection{Parameters related to the domination chain}
When we do not want a vertex can serve as the private neighbor of itself, we get an variant of $\text{IR}$-number. 
Given a graph $G$, a subset $S$ of $V(G)$ is an {\em open irredundant set} if every vertex $s$ in $S$ has at least one private neighbor outside of $S$, i.e., $P\!N[s,S]-S\neq \emptyset$. The {\em open irredundance number} $O\!I\!R(G)$ is  defined as %the order of the maximum open irredundance set in $G$. 
	$$O\!I\!R(G)=\max\{|S| : S \text{ is an open irredundant set of } G\}.$$

Clearly, $O\!I\!R(G)\le I\!R(G)$ for any graph $G$.
We also determine $B$-$O\!I\!R$ in the following.
\begin{restatable}{thm}{OIR}
\label{OIR}
	$\HH\in B$-$O\!I\!R$  if and only if 
	\(
	\HH\le \{K_{1,n}^*, K_n^*, P_n, C\!K_n, F_n\}
	\)
	for some positive integer $n$.	
\end{restatable}

Arumugam, Favaron, Sudha~\cite{afs} and   Arumugam, Subramanian~\cite{as} introduced the independence saturation number and  the irredundance saturation number.

	Let $G=(V,E)$ be a graph and $v\in V$. 
	Let $$I\!S(v)=\max\{|S|: S \text{ is an independent set of $G$ with $v\in S$}\},$$ 
	and define the {\em independence saturation number} of $G$ as 
	$$I\!S(G)=\min \{I\!S(v):v\in V\}.$$ 
	Let $$I\!R\!S(v)=\max\{|S| : S \text{ is an irredundant set in $G$ with $v\in S$}\},$$ 
	and define the {\em irredundance saturation number} of $G$ as 
	$$I\!R\!S(G)=\min \{I\!R\!S(v):v\in V\}.$$   

Thus $I\!S(v)\le I\!R\!S(v)$ as all independent sets are irredundant. Thus $I\!S(G)\le I\!R\!S(G)$. Since $I\!S(G)$ is the size of some maximal independent set, $i(G)\le I\!S(G)\le\alpha(G)$ and similarly $i\!r(G)\le I\!R\!S(G)\le I\!R(G)$.
We  determine $B$-$I\!S$ and $B$-$I\!R\!S$ as follows.

\begin{restatable}{thm}{IS}
\label{IS}
	$\HH\in B$-$I\!S$ if and only if 
	$$\HH \le \{K_{1,n}^*, K_n^*,P_n, K_{n,n}, B\!S_n^p: 2\le p\le n-3\}$$ 
	for some $n\ge 5$.	
\end{restatable}
%Denote by $I\!S_n$ the smallest constant depending only on $n$ in Theorem~\cref{IS}, i.e., $I\!S(G)<I\!S_n$ for any $\{K_{1,n}^*, K_n^*,P_n, K_{n,n}, B\!S_n^p: 2\le p\le n-3\}$-free graph $G$.
%Finally, we characterize $\HH\in B$-$I\!R\!S$.

\begin{restatable}{thm}{IRS}
\label{IRS}
	$\HH\in B$-$I\!R\!S$ if and only if 
	\[\HH \le \{K_{1,n}^*, K_n^*,P_n, K_{n,n}, C\!K_n, B\!S_n^p: 2\le p\le n-3\}\]
	for some $n\ge 5$.	
\end{restatable}

The rest of the paper is arranged as follows. In Section 2, we prove Theorem~\ref{THM: main}. We give the proofs of Theorem \ref{OIR}, \ref{IS} and \ref{IRS} in Section 3.

\section{Proof of Theorem~\ref{THM: main}}
We begin by giving some preliminaries. Notably, there exists a quality bound between the lower irredundance number ($ir(G)$) and the domination number ($\gamma(G)$).
%  inequality shows that $ir$ and $\gamma$ are bounded at the same time.  
\begin{prop}[\cite{hhs}]\label{PROP: ir}
For any graph $G$, $ir(G)\le \gamma(G)\le 2ir(G)-1$.	
\end{prop}
%Combining Theorem~\ref{domination number}, $\mathcal{H}\in B$-$\mu$ for $\mu=ir$ and $\gamma$  are the same. 
%\begin{cor}
%The following are equivalent:
%\begin{itemize}
%	\item[(1)] $\HH\in $  $B$-$ir$.	
%	\item [(2)] $\HH\in $  $B$-$\gamma$.
%	\item[(3)] $\HH \le \{K_{1,n}^*, K_n^*, P_n\}$ for some positive integer $n$.
%\end{itemize}
%\end{cor}

The following lemma has appeared in~\cite{alr}.
%In order to determine $B$-$i$, we introduce the following results first.
\begin{lem}[\cite{alr}] \label{bipartite Ramesy}
For any positive integer $n$, there exists a smallest integer $B\!R(n)$ such that every bipartite graph $G=(V_{1},V_{2}, E)$ with $|V_i|\ge B\!R(n)$ for $i= 1,2$, there are  subsets $U_i\subseteq V_i$ of order $n$ %$|U_i|=q$ 
satisfying that %$|E_G[U_1, U_2]|=|U_1||U_2| \mbox{ or } 0$. %or %$E_G[U_1, U_2]=\emptyset$.
$G[U_1,U_2]\cong K_{n,n}$ or $E_{2n}$.
\end{lem}

Let $\BB\!\mathcal S_n^p$ be the family of graphs obtained by adding additional edges connecting the leaves adjacent to one center and the leaves adjacent to anther center  of $B\!S_n^p$.
The following result was given by Zverovich and Zverovich in~\cite{zz}. 
\begin{thm}[\cite{zz}] \label{BS_n}
If a graph $G$ is $\BB\!\mathcal S_{k-1}^2$-free, where $k\ge 3$, then  \[
 i(G)\le \gamma(G)(k-2)-(k-3).
\]
\end{thm}

\begin{lem}\label{er fen}
%Let $G\in \BB \mathcal S_{BR_n}^p$, then $G$ contains an induced $BS_n^p$ or $K_{n,n}$.
$\{K_{n,n}, B\!S_n^p\}\le \BB\!\mathcal S_{B\!R(n)}^p$.	
\end{lem}
\begin{proof}
Denote by $V_i \, (i=1,2)$ the set of leaves adjacent to the corresponding two centers of $B\!S_{B\!R(n)}^p$, respectively. Then $|V_1|=|V_2|=B\!R(n)$. Recall that any graph $G\in \BB\!\mathcal S_{B\!R(n)}^p$ is obtained  by adding some edges between $V_1$ and $V_2$ from $B\!S_{B\!R(n)}^p$.
Applying \cref{bipartite Ramesy} to the bipartite graph $G[V_1,V_2]$, we have $U_i\subseteq V_i$ with $|U_i|=n$ for $i=1,2$ such that 
 $G[U_1, U_2]\cong K_{n,n}$ or $G[U_1, U_2]$ is an empty graph (thus $B\!S_n^p\prec G$ in this case). Therefore, $\{K_{n,n}, B\!S_n^p\}\le \BB\!\mathcal S_{B\!R(n)}^p$.	
\end{proof}

Now we are ready to prove Theorem~\ref{THM: main}.

\begin{proof}[Proof of Theorem~\ref{THM: main}:]
(1)	It is a corollary of \cref{domination number} and Proposition~\ref{PROP: ir}.

(2) We first prove the ``only if'' part. Suppose that there is a constant $c$ such that every $\HH$-free graph $G$ satisfies $i(G)< c$. Let $n=3c-2$. Then  $i(K_{1,n}^*)=i(K_n^*)=i(K_{n,n})=n$, $i(P_n)=c$,  and $i(B\!S_n)=n+1$. This implies that $K_{1,n}^*, K_n^*, P_n, K_{n,n}$, and $B\!S_n$  are not $\HH$-free. Therefore, $\HH \le \{K_{1,n}^*, K_n^*, P_n, K_{n,n}, B\!S_n\}$.

Next we prove the ``if'' part. Now suppose $\HH \le \{K_{1,n}^*, K_n^*, P_n, K_{n,n}, BS_n\}$. Given any connected $\HH$-free graph $G$, we know that $G$ is $\{K_{1,n}^*, K_n^*, P_n, K_{n,n}, B\!S_n\}$-free. Thus $G$ is $\{K_{1,n}^*, K_n^*, P_n\}$-free too. By \cref{domination number}, $\gamma(G)<\gamma_n$.  By \cref{er fen}, we also have that $G$ is $\BB\!\mathcal S_{B\!R(n)}^2$-free.
According to \cref{BS_n}, 
$$i(G)\le \gamma(G)(B\!R(n)-1)-(B\!R(n)-2)<\gamma_n B\!R(n).$$

(3) Recall that we need to prove the following statements are equivalent (a) $\HH\in B$-$\text{IR}$, (b) $\HH\in B$-$\Gamma$, and (c) $\HH \le \{K_{1,n}, K_n^*, P_n, CK_n\}$ for some positive integer $n$.

$(a)\Rightarrow (b)$ is clear since $\Gamma(G)\le \text{IR}(G)$. In order to prove $(b)\Rightarrow (c)$, we only need to show that $K_{1,n}, K_n^*, P_n$ and $CK_n$ have unbounded upper domination number $\Gamma$ as $n$ increases. By \cref{duli},  $K_{1,n}, K_n^*$ and $P_n$ have unbounded independence number $\alpha$, so does the upper domination number $\Gamma$. %The set of all the leaves of $K_{1,n}$ is a largest minimal dominating set of $K_{1,n}$, 
The maximum clique of $C\!K_n$ forms a largest minimal dominating set of $C\!K_n$. Therefore, we have $\Gamma(C\!K_n)=n$. 

Now we prove that $(c)\Rightarrow (a)$. We claim that, for any $\{K_{1,n}, K_n^*, P_n, CK_n\}$-free graph $G$, $\text{IR}(G)< R(R(n,n),\alpha_n)$.
Suppose not, then there exists an \text{IR}-set $S$ of $G$ such that $|S|\ge R(R(n,n),\alpha_n)$. Since $\alpha(G[S])\le \alpha(G) <\alpha_n$, there must be a clique $K\cong K_{R(n,n)}$ in $G[S]$. Take a private neighbor for each vertex in $K$, and denote by $K'$ the set of these corresponding private neighbors. Then $|K'|=|V(K)|=R(n,n)$. Hence $G[K']$ contains an induced $K_n$ or $E_n$. If $G[K']$ contains a $K_n$, then $G$ contains an induced $C\!K_n$, a contradiction. If $G[K']$ contains a $E_n$, then $G$ contains an induced $K_n^*$, a contradiction too. The proof is completed.
\end{proof}

\section{Proofs of Theorems \ref{OIR}, \ref{IS} and \ref{IRS}}
We revisit \cref{OIR} as follows.
\OIR*
\begin{proof}
The ``only if'' part can be checked directly from the fact that $O\!I\!R(K_{1,n}^*)=O\!I\!R(K_n^*)=O\!I\!R(C\!K_n)=O\!I\!R(F_n)=n$, and $O\!I\!R(P_n)=\lceil n/3\rceil$. 

Now we prove the ``if'' part. All we need to show is that if a connected graph $G$ satisfies $O\!I\!R(G)\ge R(R(n,n),R(n,(2n-1)\gamma_n)$, then $G$ contains an induced $K_{1,n}^*, K_n^*, P_n, C\!K_n$ or $ F_n$. Take an $O\!I\!R$-set $S$ of $G$, and thus $|S|\ge R(R(n,n),R(n,(2n-1)\gamma_n)$. Note that each element in $S$ has at least one private neighbor outside of $S$. For any $s\in S$, fix one of its private neighbors and denote it as $s'$. For any subset $X$ of $S$, let $X'=\{x':x\in X\}$. 

If $G[S]$ contains a clique $K=K_{R(n,n)}$, then $|K'|= |V(K)|=R(n,n)$. 
If $G[K']$ contains a $K_n$, then $G$ has an induced $C\!K_n$; otherwise, $G[K']$ contains an $E_n$, in this case, $G$ contains an induced $K_n^*$.

Now assume $G[S]$ contains a stable set $E$ of order $R(n,(2n-1)\gamma_n)$. Then $|E'|= |E|=R(n,(2n-1)\gamma_n)$.
If $G[E']$ contains a $K_n$, then $G$ has an induced $K_n^*$ too; otherwise,  $G[E']$ contains a stable set $Y'$ of order $(2n-1)\gamma_n$. 
If $\gamma (G)\ge \gamma_n$, then $G$ contains an induced $K_n^*, P_n$ or $K_{1,n}^*$ by \cref{domination number}. We are done. Thus we may assume $\gamma(G)<\gamma_n$. Then there must be a vertex $v$ in a $\gamma$-set of $G$ that dominates at least $2n-1$ vertices in $Y'$. 
Therefore, we can choose $Z'\subseteq N_G(v)\cap Y'$ with $|Z'|=2n-1$. Let $Z=\{z\in S:z'\in Z'\}$ denote the subset of $S$ corresponding to $Z'$. Then $|Z|=|Z'|=2n-1$ and $v\not \in Z\cup Z'$. If $v$ dominates $n$ vertices in $Z$, then $G$ contains an induced $F_n$. Otherwise,  there are at least $n$ vertices of $Z$ nonadjacent to $v$, thus  $G$ contains an induced $K_{1,n}^*$ in this case.
Therefore, we have that, for any connected $\HH$-free graph $G$,  
$$O\!I\!R(G)< R(R(n,n),R(n,(2n-1)\gamma_n).$$
\end{proof}

Before we prove Theorem \ref{IS} and \ref{IRS}, we first introduce a lemma of Lozin \cite{l17} to bound the matching number $\nu(G)$ of bipartite graphs.

\begin{lem}[\cite{l17}]\label{induced matching}
%For any positive integers $n$, there exists a minimum $q = q(n)$ such that every bipartite graph $G$ with a matching of size $q$ contains %either an induced matching of size $n$ or an induced $K_{n,n}$.	an induced $nK_2$ or $K_{n,n}$.
Let $G$ be an $\{nK_2,K_{n,n}\}$-free bipartite graph. Then there exists a minimum $q(n)$, such that 
$\nu(G)<q(n)$. %the matching number of $G$ is less than $q(n)$.
\end{lem}

\cref{IS} is restated below.
\IS*
\begin{proof}
We first prove the ``only if'' part. %For any vertex $v$ in $G$, $I\!S(v)\ge i(G)$ since $i(G)$ is the minimum cardinality of a maximal stable set of $G$. Hence $I\!S(G)\ge i(G)$. 
$I\!S(G)\ge i(G)$ can be arbitrarily large for $G\in\{K_{1,n}^*, K_n^*,P_n, K_{n,n}\}$ by Theorem~\ref{THM: main} (2). For graph $B\!S_n^p$, %we have $I\!S(B\!S_n^p)\ge I\!S(B\!S_n)\ge i(B\!S_n)=n+1$. 
each vertex $v$ together with the leaves adjacent to some center of $B\!S_n^p$ can form a stable set of size $n+1$. Then $I\!S (B\!S_n^p)=n+1$.
Therefore, $\HH \le \{K_{1,n}^*, K_n^*,P_n, K_{n,n}, B\!S_n^p: 2\le p\le n-3\}$ for some $n\ge 5$.	

Next we prove the ``if'' part. Let $N=(\gamma_n-1)(q(n)+2B\!R(n)-1)+2$, where $q(n)$ is defined in~\cref{induced matching} and $B\!R(n)$ is defined in~\cref{bipartite Ramesy}. We show that if a 
connected  graph $G$ satisfies $I\!S(G)\ge N$, then $G$ contains an induced $K_{1,n}^*, K_n^*, P_n,K_{n,n}$ or $ B\!S_n^p$ for some integer $p$. (We do not require $p\le n-3$ as $P_n\prec B\!S_n^p$ if $p>n-3$.)
By \cref{domination number}, we may assume that $G$ has a $\gamma$-set $D$ with $|D|< \gamma_n$, otherwise, $G$ contains an induce $K_{1,n}^*, K_n^*$, or $P_n$, we are done. 

Let $u$ be a vertex of $G$ with maximum local stability number $\alpha_l(u):=\alpha(G[N(u)])$. Choose $U\subset N_G(u)$ be a stable set of order $\alpha_l(u)$. Since $I\!S(u)\ge I\!S(G)\ge N$, we can find a stable set $V'$ of size $N-1$ such that $N_G(u)\cap V'=\emptyset$. Thus there is a vertex $v\in D$ such that $|N_G(v)\cap V'|\ge q(n)+2B\!R(n)$. 
Let $V=N_G(v)\cap V'$. Thus 
\[|U|=\alpha_l(u) \ge \alpha_l(v)\ge |V|\ge q(n)+2B\!R(n).\] 
Note that $N_G(u)\cap V=\emptyset$ and thus $u\not =v$. 

\noindent{\bf Case 1:}  $N_G(v)\cap (U\cup \{u\})\neq \emptyset$.\\
Let $U_1=N_G(v)\cap U$ and $U_2=U-U_1$.

\begin{claim}\label{CL: U2}
$|U_2|\ge B\!R(n)$.	
\end{claim}
\begin{proof}
Apply K\"onig Theorem to the bipartite graph $G_1=G[U_1,V]$, we have 
\[\alpha(G_1)+\nu(G_1)=|U_1|+|V|.\]
 % where $\nu(G)$ is the matching number of a graph $G$.
If $\nu(G_1)\ge q(n)$, then, by~\cref{induced matching}, $nK_2\prec G[U_1,V]$.  Thus the $nK_2$ in $G[U_1,V]$ together with $u$ forms an induced $K_{1,n}^*$ in $G$, we are done. Therefore, we may assume 
$\nu(G_1)<q(n)$. 
Since $\alpha_l(u)$ is maximum, we have 
\[
0\ge \alpha_l(v)-\alpha_l(u)\ge \alpha(G[U_1, V])-|U|> (|U_1|+|V|-q(n))-(|U_1|+|U_2|). 
\]
Therefore, $|U_2|> |V|-q(n)\ge B\!R(n)$. 
\end{proof}

If $uv\in E(G)$, then some graph in $\BB\!\mathcal S_{B\!R(n)}^2$ appears in $G[\{u,v\}\cup U_2\cup V]$. By~\cref{er fen}, $G$ contains an induced $K_{n,n}$ or $B\!S_n$. 
Now suppose $uv\notin E(G)$. If there exists a vertex $p\in U$ satisfying $pv\in E(G)$ and $|N_G(p)\cap V|\ge B\!R(n)$, then $\BB\!\mathcal S_{B\!R(n)}^2$ appears in $G[\{u,p\}\cup (U\setminus\{p\}) \cup (N_G(p)\cap V)]$.
Thus for any $p\in U$ with $pv\in E(G)$, we have $|V\setminus N_G(p)|= |V|-|N_G(p)\cap V|\ge B\!R(n)$. Then $\BB\!\mathcal S_{B\!R(n)}^3$ appears in $G[\{u,p,v\}\cup U_2 \cup (V\setminus N_G(p))]$.
In each case, we are done  from~\cref{er fen}.

\vspace{5pt}

\noindent{\bf Case 2.} $N_G(v)\cap ( U\cup \{u\})=\emptyset$. 

If there exist $p\in U$ and $q\in V$ satisfying $pq\in E(G)$, then 
$\max(|N_G(p)\cap V|, |N_G(q)\cap U|)<B\!R(n)$, otherwise,  $\BB\!\mathcal S_{B\!R_n}^2$ appears in 
$$G[\{u,p\}\cup (U\setminus\{p\}) \cup (N_G(p)\cap V)] \text{ or } G[\{v,q\}\cup (N_G(q)\cap U)\cup (V\setminus\{q\})].$$
Thus $\BB \mathcal S_{B\!R(n)}^4$ appears in $G[\{u,p,q,v\}\cup (U-N_G(q))\cup (V-N_G(p))]$ as 
$$\min(|U-N_G(q)|,|(V-N_G(p)|)\ge BR(n).$$
Therefore,  we are done  from~\cref{er fen} in each case.

Finally assume $E_G[\{u\}\cup U, \{v\}\cup V]=\emptyset$.
Let $P=x_1\dots x_p$ be a shortest path connecting $\{u\}\cup U$ and $\{v\}\cup V$ such that $N_G(x_1)\cap (\{u\}\cup U)\not =\emptyset$ and $N_G(x_p)\cap (\{v\}\cup V)\not =\emptyset$. It is well-defined as $E_G[\{u\}\cup U, \{v\}\cup V]=\emptyset$.
If $\max\{|U\setminus N_G(x_1)|, |V\setminus N_G(x_p)|\}< n$, then $x_1\not =x_p$, otherwise, we have $\alpha_l(x_1)>\alpha_l(u)$, a contradiction to the assumption of $u$.
Thus we can get an induced $B\!S_n^p$ from $\{x_1,\dots ,x_p\}\cup (N_G(x_1)\cap U)\cup (N_G(x_p)\cap V)$. By symmetry, we assume $|U\setminus N_G(x_1)|\ge n$. If $x_1u\in E(G)$, then one can find a bistar $B\!S_n^q$ ($q\ge p+1$) with one center $u$ and its leaves coming from $U\setminus N_G(x_1)$. Otherwise, choose $x_0\in U$ with $x_1x_0\in E(G)$. We also have a bistar $B\!S_n^q$ ($q\ge p+2$) with one center $u$ and its leaves coming from $U\setminus (N_G(x_1)\cup\{x_0\})$. (The another center of the $B\!S_n^q$ will be  $x_p$ or $v$ according to whether $|V\setminus N_G(x_p)|<n$.) 

\end{proof}

Denote by $I\!S_n$ the smallest constant depending only on $n$ in \cref{IS}, i.e., $I\!S(G)<I\!S_n$ for any $\{K_{1,n}^*, K_n^*,P_n, K_{n,n}, B\!S_n^p: 2\le p\le n-3\}$-free graph $G$.
Now we restate \cref{IRS} as follows.

\IRS*

\begin{proof}
We first prove the ``only if'' part. Recall that $I\!R\!S(G)\ge I\!S(G)$ for any graph $G$. Thus, by~\cref{IS},  $I\!R\!S(G)$ is unbounded for any $G\in \{K_{1,n}^*, K_n^*,P_n, K_{n,n}, B\!S_n^p: 2\le p\le n-3\}$ as $n$ increases.
Note that each vertex in $C\!K_n$ belongs to an irredundant set of order $n$. Thus $I\!R\!S(C\!K_n)=n$. Therefore, $\HH \le \{K_{1,n}^*, K_n^*,P_n, K_{n,n}, C\!K_n, B\!S_n^p: 2\le p\le n-3\}$ for some $n\ge 5$.	

Next we prove the ``if'' part. All we need to show is that if a connected graph $G$ satisfies $I\!R\!S(G)\ge I\!S_n+R(R(n,n),R(n,I\!S_n+n))$, then $G$ contains an induced $K_{1,n}^*, K_n^*, P_n, K_{n,n}$, $C\!K_n$, or $B\!S_n^p$ for some positive integer $p$. By \cref{IS}, we may assume that there is a vertex $v\in V(G)$ such that $I\!S(v)=I\!S(G)< I\!S_n$. %(otherwise, $G$ must contains an induced subgraph $H\in \{K_{1,n}^*, K_n^*,P_n, K_{n,n}, B\!S_n^p: 2\le p\le n-3\}$). 

Let $S$ be an $I\!R\!S(v)$-set. Thus $v\in S$. Let $L$ be the set of isolated vertices in $G[S]$. Then $L\cup \{v\}$ is a stable set. Thus $|L\cup \{v\}|\le I\!S(v)<I\!S_n$. Let $X=S\setminus (L\cup\{v\})$.
Then every vertex in $X$ has at least one private neighbor outside of $S$. For any $x\in X$, fix one of its private neighbor as $x'$.
Let $X'=\{x':x\in X\}$. 
Then $$|X'|= |X|=|S|-|L\cup \{v\}| \ge I\!R\!S(G)-I\!S_n\ge R(R(n,n),R(n,I\!S_n+n)).$$

If $G[X']$ contains a clique $K'$ of order $R(n,n)$, then $G$ contains an induced $K_n^*$ or $C\!K_n$ by applying Ramsey's Theorem to $G[K]$, where $K=\{x\in X:x'\in K'\}$. 
Now we assume that $G[X']$ contains a stable set $Y'$ of order $R(n,I\!S_n+n)$. Let $Y=\{x\in X: x'\in Y'\}$. Applying Ramsey's Theorem to $G[Y]$, we have $G[Y]$ contains either a $K_n$ or a stable set $Z$ of order $I\!S_n+n$.    
The former case implies that $G$ contains an induced $K_n^*$. For the latter, since $I\!S(v)<I\!S_n$, $|N_G(v)\cap Z|\ge n$. Therefore, $G$ contains an induced $K_{1,n}^*$ in this case.
\end{proof}

\noindent{\bf Data Availability}\\
 Data sharing not applicable to this article as no datasets were generated or analysed
during the current study.

%\noindent{\bf Declarations}
\noindent{\bf Conflict of interest:} No conflict of interest exits in the submission of this manuscript, and manuscript is
approved by all authors for publication. I would like to declare on behalf of my co-authors that the work
described was original research that has not been published previously, and not under consideration for
publication elsewhere, in whole or in part. All the authors listed have approved the manuscript that is
enclosed.

\noindent{\bf Acknowledgments:} This work was supported by the National Key Research and Development Program of China (2023YFA1010200), the National Natural Science Foundation of China (No. 12071453),  and the Innovation Program for Quantum Science and Technology (2021ZD0302902).

\end{document}